\documentclass[10pt]{amsart}
\usepackage{latexsym}
\usepackage{amsmath}
\usepackage{amsfonts}
\usepackage{amsthm}
\usepackage{amssymb}
\usepackage[T1]{fontenc}
\begin{document}

\title[Sincov's inequalities]{Sincov's inequalities on topological spaces}

\author{W\l{}odzimierz~Fechner}
\address{Institute of Mathematics, Lodz University of Technology, ul. W\'olcza\'nska 215, 90-924 \L\'od\'z, Poland}
\email{wlodzimierz.fechner@p.lodz.pl}

\newtheorem{thm}[]{Theorem}
\newtheorem{cor}[]{Corollary}
\newtheorem{lem}[]{Lemma}
\newtheorem{prop}[]{Proposition}
\theoremstyle{remark}
\newtheorem{rem}[]{Remark}
\newtheorem{ex}[]{Example}
\newtheorem{pro}[]{Problem}
\newcommand{\N}{\mathbb{N}}
\newcommand{\R}{\mathbb{R}}
\newcommand{\C}{\mathbb{C}}
\newcommand{\Q}{\mathbb{Q}}
\newcommand{\eps}{\varepsilon}
\newcommand{\f}{\varphi}
\renewcommand{\(}{\left(} \renewcommand{\)}{\right)}
\renewcommand{\[}{\left[} \renewcommand{\]}{\right]}


\begin{abstract}
Assume that $X$ is a non-empty set and $T$ and $S$ are real or complex mappings defined on the product $X \times X$.
Additive and multiplicative Sincov's equations are:
$$T(x,z) = T(x, y ) + T(y, z)$$
and 
$$S(x,z) = S(x, y ) \cdot S(y, z),$$
respectively. Both equations play important roles in many areas of mathematics. In the present paper we study related inequalities. We deal with functional  inequality
$$
G(x,z) \leq G(x, y ) \cdot G(y, z), \quad  x , y, z \in X
$$
and we assume that $X$ is  a topological space and $G\colon X \times X \to \R$ is a continuous mapping.
In some our statements a considerably weaker regularity than continuity of $G$ is needed.
 We also study the reverse inequality:
$$F(x,z) \geq F(x, y ) \cdot F(y, z), \quad  x , y, z \in X  $$
and the additive inequality (the triangle inequality):
$$H(x,z) \leq H(x,y) + H(y,z), \quad x, y , z \in X.$$
A corollary for generalized (non-symmetric) metric is derived.
\end{abstract}

\maketitle

\section{Introduction}

Throughout the paper it is assumed that $\R$ denotes the set of real numbers, $\Q$ is the set of rationals and $\N $ stands for the set of positive integers.  
Moreover, for $a, b$ from $\R$ or from $\R\cup \{-\infty, + \infty \}$ respectively, open, closed and half-open intervals with endpoints $a$ and $b$ are denoted by $(a,b)$, $[a,b]$,  $[a,b)$ and $(a,b]$, respectively.

\medskip

Assume that $X$ is a non-void set and $S \colon X \times X \to \R$ is an arbitrary mapping. 
By multiplicative Sincov's equation we mean
\begin{equation}
\label{eq}
 S(x,z) = S(x, y ) \cdot S(y, z), \quad  x , y, z \in X.
\end{equation}
The general solution of \eqref{eq} is given by $S =0$ on $X \times X$ or there exists a function $f\colon X \to \R\setminus \{0\}$ such that
\begin{equation}
\label{sol}
S(a,b) = \frac{f(a)}{f(b)}, \quad a, b \in X
\end{equation}
(see D. Gronau \cite[Theorem]{G1}). Equation \eqref{eq} is of significant importance and its history goes back to XIX century; for more information we refer the reader to works by D. Gronau \cite{G1, G2}. A connection of Ulam-type stability of equation \eqref{eq} with some generalizations of the Cauchy-Schwarz inequality was observed in \cite{ja}.

\begin{rem}
Directly from the representation \eqref{sol} of solutions of Sincov's equation one can easily observe, that if two given solutions of \eqref{eq} defined on the same set $X$ are comparable, then they are equal. In particular, it makes no sense to speak on maximal or minimal solutions of  Sincov's equation. This observation is important in the light of our subsequent results (see Corollaries \ref{csup}, \ref{csupF}, \ref{cT} below), in which we provide representations of solutions of inequalities as a pointwise supremum or infimum of a certain family of solutions of equations. 
\end{rem}

\section{Multiplicative Sincov's inequality}

In this section we will study the following functional inequality: 
\begin{equation}
\label{main}
G(x,z) \leq G(x, y ) \cdot G(y, z), \quad  x , y, z \in X,
\end{equation}
which will be called multiplicative Sincov's inequality.
In our main results we will assume that $X$ is a topological space and $G\colon X \times X \to \R$ is continuous, or it satisfies a weaker regularity condition. We will prove that, either $G$ is in a sense trivial solution, or there is a map which lies below $G$ on $X\times X$, is equal to $G$ at a given point $(x_0,y_0)\in X \times X$ and solves Sincov's equation \eqref{eq}. From this we will derive a representation of solutions of \eqref{main} as supremum of functions of the form \eqref{sol}. 

\medskip

We begin with sorting out two classes of solutions, namely non-positive mappings and mappings whose image lies in a compact interval of positive reals.

\begin{ex}\label{e1}
For arbitrary non-void set $X$ every map  $G\colon X \times X \to (-\infty, 0]$ and  every map  $G\colon X \times X \to [c, c^2]$ with some $c\geq 1$ yield a solution of \eqref{main}. 
\end{ex}

\begin{prop}\label{t1}
Assume that $X$ is a connected topological space and  $G\colon X \times X \to \R$ is a continuous solution of \eqref{main}. If $G$ attains a non-positive value, then $G$ is non-positive on $X \times X$.
\end{prop}
\begin{proof}
By assumption there exists a $(a_0,x_0) \in X \times X$ such that $G(a_0,x_0)\leq 0$.
Suppose for the contrary that $G(a_1,x_1)>0$ for some $(a_1,x_1) \in X \times X$. We can find a $(a_2,x_2) \in X \times X$ such that $G(a_2,x_2)=0$ (consider the sign of $G(a_0, x_1)$ and apply the continuity of one of the mappings $G(a_0, \cdot )$ or $G(\cdot, x_1 )$). 
 From \eqref{main} we derive 
 $$G(a_1, x_2)\leq G(a_1,a_2) G(a_2,x_2)= 0.$$
Next, by continuity used once more, we obtain the existence of some $x_3\in X$ such that $G(a_1, x_3)=0$. Consequently,  $$ 0 < G(a_1,x_1) \leq G(a_1, x_3) G(x_3, x_1) = 0;$$ a contradiction.
\end{proof}

Let us denote by (Z) the following property of a function $f\colon X \to \R$ defined on a non-void set:
\begin{enumerate}
	\item[(Z)] \emph{if there exist $x, y \in X$ such that $f(x)\leq 0\leq f(y)$, then there exists $z \in X$ such that $f(z)=0$.}
\end{enumerate}
Clearly, every continuous mapping on a  connected topological space has property (Z).

In Proposition \ref{t1} it is enough to assume that each section of $G$ has property (Z); in particular, no topology on $X$ is needed and the proof remains unchanged. We will provide some examples illustrating the situation. 

\begin{ex}\label{r1}
Let $A\colon \R \to \R$ be a discontinuous additive function with connected graph. Such functions do exist, see e.g. L. Sz\'{e}kelyhidi \cite{Sz} and necessarily have the Darboux (intermediate value) property. Take  $X=\R$ and define $$G_1(a,b) = \exp (A(a)-A(b)) , \quad a, b \in \R.$$ 
Then $G_1$ is a discontinuous solution of \eqref{main} (in fact, it is a solution of \eqref{eq}) with all sections having the Darboux property. 

Let us modify the above mapping a bit. Define $X=\{(x,A(x)) : x \in \R\}$ and $$G_2((a,A(a)),(b, A(b))) = \exp (A(a)-A(b)) , \quad a, b \in \R.$$ Note that this time $G_2$ is continuous and $X$ is a connected space.

Finally, let  $X$ be a disconnected topological space. Define $G_3(a,b)$ as being equal  to $1$ whenever $a,b$ lies in the same connected component of $X$ and $-1$ elsewhere. It is easy to check that $G_3$ is a continuous solution of \eqref{main}. Therefore, the assumption that $X$ is connected cannot be dropped.
\end{ex}

From now on we will focus on non-negative solutions of \eqref{main}. First we will make an easy observation that a special case of Proposition \ref{t1} with $G$ attaining non-negative values remains valid  without any additional assumptions.

\begin{prop}\label{p0}
Assume that $X$ is a non-void set and  $G\colon X \times X \to [0, + \infty)$ is a solution of \eqref{main}. If $G$ has a zero, then $G=0$  on $X \times X$.
\end{prop}
\begin{proof}
Let $(a_0,x_0) \in X \times X$ be such that $G(a_0,x_0)=0$. Then for arbitrary $a, b \in X$ we have
 $$0 \leq G(a,b)\leq G(a,a_0) G(a_0,b)\leq G(a,a_0) G(a_0, x_0)G(x_0, b)  = 0.$$
\end{proof}

In view of Proposition \ref{t1} and Proposition \ref{p0}, from now on we will study positive solutions only. First, we will show that in case of positive and bounded solutions there is an estimate from below by a positive number.  

\begin{prop}\label{b}
Assume that $X$ is a non-void set and  $G\colon X \times X \to (0, + \infty)$ is a bounded solution of \eqref{main}. Then there exists some $c\geq 1$ such that $G(X \times X )\subseteq [{1}/{c} , c]$.
\end{prop}
\begin{proof}
First, observe that for all $ a \in X$ directly from \eqref{main} we get
$$G(a,a) \leq G(a,a) G(a,a).$$
Therefore, since $G$ is positive, then $G(a,a)\geq 1$ for every $a \in X$.
Next, define $c=\sup \{G(a,b) : a, b \in X\}$. We have
$$1 \leq G(a,a) \leq G(a,b) G(b,a)\leq  G(a,b) c, \quad a, b \in X.$$ Therefore, $\inf \{G(a,b) : a, b \in X\} \geq {1}/{c}$. 
\end{proof}


\begin{prop}\label{p1}
Assume that $X$ is a non-void set and  $G\colon X \times X \to (0, + \infty)$ is a solution of \eqref{main}. 
Then the following estimate holds true:
\begin{equation}\label{est}
\frac{1}{G(y,x)}\leq \frac{G(a,y)}{G(a,x)} \leq G(x,y), \quad a, x, y \in X.
\end{equation}
\end{prop}
\begin{proof}
Estimate \eqref{est} follows immediately from \eqref{main} applied twice. 
\end{proof}

Let us associate with $G\colon X \times X \to \R$ a map $G^*\colon X \times X \to \R$ given by
\begin{equation}\label{star}
G^*(x,y) = G(y,x), \quad x, y \in X.
\end{equation}
It is clear that $G$ solves \eqref{main} if and only if $G^*$ solves \eqref{main}. 

\begin{rem}
An analogue of Proposition \ref{p1} with the roles of variables reversed is true, as well. More precisely, we have
\begin{equation}\label{est*}
\frac{1}{G(x,y)}\leq \frac{G(y,a)}{G(x,a)} \leq G(y,x), \quad a, x, y \in X.
\end{equation}
It is enough to consider map $G^*$ defined by \eqref{star} and apply Proposition \ref{p1}. On the other hand, an easy example shows that it is possible that a solution $G$ of \eqref{main} has every left section bounded and at the same time every right section unbounded, or conversely. Indeed, consider $X= (1, + \infty)$ and take $G(x,y) = xy^{-1}$ for $x, y \in (1, + \infty)$ (see \cite[Example 2]{ja}). 
\end{rem}

We will denote the diagonal of the product $X\times X$ as
$\Delta = \{ (x,x) : x \in X\}.$

\begin{lem}\label{sup}
Assume that $X$ is a non-void set and  $G\colon X \times X \to (0, + \infty)$ is a solution of \eqref{main}. 
Then there exists a map $\widetilde{G}\colon X \times X \to (0, + \infty)$ which enjoys the following properties:
\begin{itemize}
	\item[(i)] ${1}/{G^*}\leq \widetilde{G}\leq G$ on $X \times X$,
	\item[(ii)] $\widetilde{G} = 1$ on $\Delta$, 
	\item[(iii)] if $G = 1$ on $\Delta$, then $G=\widetilde{G}$,
	\item[(iv)] $\widetilde{G}$ solves inequality \eqref{main},
	\item[(v)] if $X$ is a topological space and $G$ is continuous, then $\widetilde{G}$ is lower semi-continuous,
	\item[(vi)] if $X$ is a topological space and the family $\{G(x,\cdot ) : x \in X \} $ is pointwise equi-continuous, then $\widetilde{G}$ is continuous at every point of $\Delta$.
\end{itemize}
\end{lem}
\begin{proof}
Define $\widetilde{G}\colon X \times X \to (0, + \infty)$ by
$$\widetilde{G}(x,y) = \sup \left\{  \frac{G(a,y)}{G(a,x)} : a \in X \right\}, \quad x, y \in X.$$
Due to estimate \eqref{est} the definition is correct and property (i) is fulfilled. 

Part (ii) is obvious. 

To prove (iii) note that if $G=1$ on $\Delta$, then by \eqref{est}
$$\widetilde{G}(x,y) \geq \frac{G(x,y)}{G(x,x)} = G(x,y), \quad x, y \in X.$$
The converse inequality follows from (i).

To justify (iv) fix arbitrary $ x, y, z, a \in X$. Directly from the definition of $\widetilde{G}$ one has
$$\widetilde{G}(x,y) \widetilde{G}(y,z) \geq  \frac{G(a,y)}{G(a,x)} \cdot \frac{G(a,z)}{G(a,y)}   = \frac{G(a,z)}{G(a,x)}$$
and (iv) follows by passing to the supremum with $a$ on the right-hand side.

Point (v) is obvious. 

To prove (vi) fix a $y\in X$ and $\eps>0$.
Let $U\subset X$ be a neighbourhood of $y$ such that for every $y_1, y_2 \in U$ and for all $x\in X$ one has 
$$\left|\frac{G(x,y_1)}{G(x,y_2)} - 1\right| < \frac{\eps}{2}.$$
For fixed $y_1, y_2 \in U$ there exists some $x_0 \in X$ such that
$$\frac{G(x_0,y_1)}{G(x_0,y_2)}> \sup \left\{  \frac{G(x,y_1)}{G(x,y_2)} : a \in X \right\} - \frac{\eps}{2} = \widetilde{G}(y_2,y_1) - \frac{\eps}{2}.$$
Join these estimates to get
$$|\widetilde{G}(y_2,y_1) - \widetilde{G}(y,y)|= |\widetilde{G}(y_2,y_1) - 1|< \eps.$$
\end{proof}

\begin{lem}\label{l1}
Assume that $X$ is a countable set, $(a_n) \subset X$ is an arbitrary sequence and
$G\colon X \times X \to (0, + \infty)$ is a solution of \eqref{main}.
Then there exists a sequence $(\alpha_n) \subset X$ such that $(\alpha_n)$ is a subsequence of $(a_n)$, the following limit exists:
\begin{equation}\label{lim}
S(b, a) = \lim_{n\to \infty}\frac{G(\alpha_n,a)}{G(\alpha_n,b)}
\end{equation}
for every $a, b \in X$ and map $S\colon X \times X \to (0, + \infty)$ defined by \eqref{lim} solves \eqref{eq}.
\end{lem}
\begin{proof}
Let $\{q_k : k \in \N\}$ be an arrangement of $X\times X$ into a sequence. We will construct an auxiliary family of sequences associated to each $q_k$. Put $a^0_n :=a_n$ for $n \in \N$. 
Next, fix a $k \in\N$, assume that sequence $(a^{k-1}_n)$ is already defined and denote $(a^k,b^k)=q_k$. From Proposition \ref{p1} we know that the values ${G(a^{k-1}_n,a^k)}/{G(a^{k-1}_n,b^k)}$ for all $n \in \N$ lie in a compact interval. Consequently, there exists a sub-sequence $(a^k_n)$ of $(a^{k-1}_n)$ with the property that the sequence $\left({G(a^k_{n},a^k)}/{G(a^k_{n},b^k)}\right)$ is convergent. 

We constructed inductively a countable family of sequences $\{(a^k_n) : k \in \N \}$ with the property that every sequence $(a^{k+1}_n)$ is a sub-sequence of $(a^{k}_n)$ and the sequences $\({G(a^k_n,a^k)}/{G(a^k_n,b^k)}\)$ are convergent, where $(a^k,b^k) = q_{k}$. 
Now, define the sequence $(\alpha_n)$ by $\alpha_n=a^n_{n}$ for $n \in \N$. To see that formula \eqref{lim} holds true for every $a, b \in X$, consider $q_k = (a,b)$ and note that the sequence $(\alpha_n)$ is from a certain moment a subsequence of  $(a^k_n)$. Having \eqref{lim}, it is straightforward to check that \eqref{eq} holds true. 
\end{proof}

\begin{cor}\label{c1}
 Assume that $X$ is a countable set, $(x_0,y_0) \in X \times X$ is an arbitrary point and $G\colon X \times X \to (0, + \infty)$ is a solution of \eqref{main}. Then there exists a function $S\colon X \times X \to (0, + \infty)$ such that $S$ is a solution of \eqref{eq}, $S(x_0,y_0) = \widetilde{G}(x_0,y_0)$, where $\widetilde{G}$ is postulated by Lemma \ref{sup}, and
\begin{equation}\label{SG}
\frac{1}{G^*}\leq S \leq G \textrm{ on } X\times X.
\end{equation}
\end{cor}
\begin{proof}
Take as $(a_n)$ a sequence such that 
$$   \lim_{n\to \infty} \frac{G(a_n,y_0)}{G(a_n,x_0)} = \sup \left\{  \frac{G(a,y_0)}{G(a,x_0)} : a \in X \right\} .$$
Now, the assertion follows from  Lemma \ref{l1}  and Proposition \ref{p1}.
\end{proof}

With the aid of results of this section, now we are able to prove the following.

\begin{thm}\label{t2}
Assume that $X$ is a separable topological space, $(x_0,y_0) \in X \times X$ is an arbitrary point and $G\colon X \times X \to (0, + \infty)$ is  a solution of \eqref{main} such that $G$ is continuous and equal to $1$ at every point of $\Delta  = \{ (x,x) : x \in X\}  $. 
Then there exists a function $S\colon X \times X \to (0, + \infty)$ such that $S$ is a solution of \eqref{eq}, $S(x_0,y_0)=G(x_0,y_0)$ and estimate \eqref{SG} is satisfied.
Moreover, $S$ is given by formula \eqref{lim} on $X \times X$ with some sequence $(\alpha_n)\subset X$.
\end{thm}
\begin{proof}
Equality  $G=\widetilde{G}$ follows from part (iii) of Lemma \ref{sup}. 
Let $X_0$ be a countable dense subset of $X$ such that $x_0,y_0 \in X_0$. Corollary \ref{c1} applied for $X_0$ and $G$ gives us a sequence $(\alpha_n)\subset X_0$ and a map $S_0$ defined on $X_0 \times X_0$ by formula 
\begin{equation}\label{lim0}
S_0(b, a) = \lim_{n\to \infty}\frac{G(\alpha_n,a)}{G(\alpha_n,b)}
\end{equation}
which satisfies $S_0(x_0,y_0)=G(x_0,y_0)$.
We will justify that the definition of mapping $S\colon X \times X \to (0, + \infty)$ via formula \eqref{lim} is correct for all $(b,a) \in X \times X$ (i.e. the limit always exists). Fix some $a, b \in X$ and take $a', b' \in X_0$ sufficiently close to $a$ and $b$. From \eqref{est} we obtain
$$\frac{1}{G(a',a)}\leq \frac{G(x,a')}{G(x,a)} \leq G(a,a'), \quad x \in X.$$
This means that the middle term is as close to $1$ as desired since $G$ is continuous and equal to $1$ at $(a,a)$. 
A similar estimate holds for $b$ and $b'$.
On the other hand, we have
$$\frac{G(\alpha_n,a)}{G(\alpha_n,b)}= \frac{G(\alpha_n,a)}{G(\alpha_n,a')}\cdot \frac{G(\alpha_n,a')}{G(\alpha_n,b')}\cdot \frac{G(\alpha_n,b')}{G(\alpha_n,b)}.$$ 
Note that first and third fractions are close to $1$, whereas by \eqref{lim0} the middle one tends to $S_0(b',a')$. This justifies our claim.
 By \eqref{est} estimate \eqref{SG} holds true on $X$. 
\end{proof}

\begin{rem}
Second part of the above proof, showing that  $S$ is well-defined on $X \times X$, proves a fact which is interesting on its own. Namely, if  $X$ is a topological space and $G\colon X \times X \to (0, + \infty)$ is  a solution of \eqref{main} such that $G$ is continuous and equal to $1$ at every point of $\Delta$, then the family
$\{ G(x, \cdot ) : x \in X\}$ is equi-continuous. This is the converse statement of part (vi) of Lemma \ref{sup}.
\end{rem}

\begin{ex}
Assumption that $G=1$ on $\Delta$ cannot be omitted. Take $X=[1, + \infty)$ and define $G\colon X \times X \to (0, + \infty)$ by
$$G(a,b) = a+b, \quad a, b \in X.$$
In particular, for $(x_0,y_0)=(1,1)$ there is no sequence $(\alpha_n) \subset X$ such that function $S$ defined by \eqref{lim} satisfies $S(1,1) = G(1,1)$.

What is more, there exist non-measurable solutions of \eqref{eq} which satisfies \eqref{SG} together with $G$ as above.
Let $A\subset [1, + \infty)$ be a non-measurable set and define
$S\colon X \times X \to (0, + \infty)$  by
$$S(a,b) = \frac{b+\chi_A(b)}{a+\chi_A(a)}, \quad a, b \in X.$$
Then $S$ is a non-measurable solution of \eqref{eq} and estimate \eqref{SG} is satisfied by $G$ and $S$.
\end{ex}

Let us introduce a class of  functions of one variable associated with a map $G\colon X \times X \to (0, + \infty)$.
$$\mathcal{G}(G) = \left\{  f\colon X \to (0, + \infty) : \forall_{ x, y \in X} \frac{f(x)}{f(y)} \leq G(x,y) \right\}.$$

\begin{cor}\label{csup}
Assume that $X$ is a separable topological space and  $G\colon X \times X \to (0, + \infty)$ is  a solution of \eqref{main} such that $G$ is continuous and equal to $1$ at every point of  $\Delta = \{ (x,x) : x \in X\} $. Then 
\begin{equation}\label{repG}
G(a,b) = \sup \left\{ \frac{f(a)}{f(b)} : f \in \mathcal{G}(G) \right\}, \quad a, b \in X.
\end{equation}
Conversely, for an arbitrary family $\mathcal{G}$ of positive functions on $X$  every mapping $G\colon X \times X \to (0, + \infty)$ defined by \eqref{repG} solves \eqref{main}, it is equal to $1$ on $\Delta$ and $\mathcal{G}\subseteq \mathcal{G}(G)$.
\end{cor}
\begin{proof}
Apply Theorem \ref{t2} for points $(x_0,y_0)$ running through the space $X \times X$ and use the form \eqref{sol} of solutions of \eqref{eq}. 

To justify the converse statement fix arbitrary $x, y, z \in X$ and $\eps>0$. There exists some $f \in \mathcal{G}$ such that 
$G(x,y) < {f(x)}/{f(y)} + \eps$. From this we have 
$$G(x,z) < \frac{f(x)}{f(y)}\cdot \frac{f(y)}{f(z)} + \eps \leq G(x,y) G(y,z) + \eps$$
and the assertion follows.
\end{proof}

\section{Second multiplicative Sincov's inequality}

The case of the reverse inequality to \eqref{main}, i.e. the inequality 
\begin{equation}
\label{F}
F(x,z) \geq F(x, y ) \cdot F(y, z), \quad  x , y, z \in X
\end{equation}
is not fully symmetric to \eqref{main}, but in some situations it can be reduced to \eqref{main}. 
First, we list some examples.

\begin{ex}\label{e2}
Functions $G_1$ and $G_2$ of Example \ref{r1} are solutions of \eqref{F} since they solve \eqref{eq}. If the topological space $X$ consists of precisely two connected components, then $G_3$ is solution of equation \eqref{eq}, as well (and thus solves \eqref{F}).
\end{ex}

\begin{ex}\label{e21}
Function $F_1\colon \R \times \R \to [0,1]$ given by $$F_1(a,b)=\chi_\Q (a-b), \quad a, b \in \R,$$  function $F_2\colon X \times X \to [0,1]$ given by $$F_2(a,b)=\chi_A(a)\cdot \chi_B(b), \quad a, b \in  X,$$ where $X$ is a non-void set and $A, B \subset X$ are arbitrary subsets, 
and function $F_3\colon [1,+\infty) \times [1,+\infty) \to [0,+\infty)$ given by $$F_3(a,b)=\frac{a-1}{b}, \quad a, b \in [1,+\infty) $$ 
all are solutions of \eqref{F}.
\end{ex}

Note that if $F$ is positive and solves \eqref{F}, then map $G=1/F$ solves \eqref{main}. 
Next, introduce a class of functions associated with a map $F\colon X \times X \to (0, + \infty)$.
$$\mathcal{F}(F) = \left\{  f\colon X \to (0, + \infty) : \forall_{ x, y \in X} \frac{f(x)}{f(y)} \geq F(x,y) \right\}.$$
From Corollary \ref{csup} we derive the following description of solutions of \eqref{F}.

\begin{cor}\label{csupF}
Assume that $X$ is a separable topological space  and  $F\colon X \times X \to (0, + \infty)$ is a solution of \eqref{F} which is continuous and equal to $1$ at every point of  $\Delta = \{ (x,x) : x \in X\} $. Then
\begin{equation}\label{repF}
F(a,b) = \inf \left\{ \frac{f(a)}{f(b)} : f\in  \mathcal{F}(F) \right\}, \quad a, b \in X.
\end{equation}
Conversely, for an arbitrary family $\mathcal{F}$ of positive functions on $X$  every mapping $F\colon X \times X \to (0, + \infty)$ defined by \eqref{repF} solves \eqref{F}, it is equal to $1$ on $\Delta$ and $\mathcal{F}\subseteq \mathcal{F}(F)$.
\end{cor}

It remains to consider the case when $F$ attains a non-positive value. 

\begin{prop}\label{Fsp}
Assume that $X$ is a non-void set and  $F\colon X \times X \to \R$ is a solution of \eqref{F}. If for every $x, y \in X$ at least one of the mappings $F(x,\cdot )$, $F(\cdot, y )$ has property {\rm (Z)}, then $F$ is non-negative on $X \times X$.
\end{prop}
\begin{proof}
Directly from \eqref{F} applied for $y=z=x$ we obtain $$F(x, x) \in [0,1], \quad  x \in X.$$
Now, suppose that $F(a_1,x_1)<0$ for some $a_1, x_1 \in X$. By \eqref{F} we have
$$F(a_1,x_1) \geq F(a_1,a_1)F(a_1,x_1),$$ thus $F(a_1,a_1)=1$.  Similarly we get $F(x_1,x_1)=1$.

Apply property (Z) for sections of $F$ crossing the point $(a_1,x_1)$ to deduce that, either there exists some $x_2 \in X$ such that $F(a_1,x_2)=0$, or there exists some $a_2 \in X$ such that $F(a_2,x_1)=0$.
Utilizing this we obtain $$0> F(a_1,x_1) \geq F(a_1,y_2)F(y_2,x_1)=0,$$
where $y_2 \in \{a_2, x_2\}$; a contradiction.
\end{proof}

\begin{lem}\label{FZ}
Assume that $X$ is a topological space  and  $F\colon X \times X \to [0, + \infty)$ is a continuous solution of \eqref{F}. Suppose that the set 
$$Z= \{ (x,y ) \in X \times X : F(x,y) = 0\}$$
of zeros of $F$ is non-empty and $(a,b) \in Z$ is arbitrary.
Then $\{a\} \times X \subseteq Z$ or $ X\times \{b\}\subseteq Z$ or there exist open non-void sets $U_1, U_2 \subset X$ such that 
 $U_1 \times \{ b\} \cup \{a\} \times U_2 \subseteq Z$.
\end{lem}
\begin{proof}
Directly from \eqref{F} we have 
$$0 = F(a,b) \geq F(a,x) F(x,b), \quad x \in X.$$
Thus, taking into account the fact that we assume that $F$ is non-negative, we obtain an alternative:
\begin{equation}\label{alt}
\forall_{x \in X} [(a,x) \in Z \textrm{ or } (x,b) \in Z].
\end{equation}
Assume that none of the sets $\{a\} \times X$ and $ X\times \{b\} $ is contained in $Z$. Thus, there exist two points, say $x_1, x_2 \in X$ such that for $x_1$ the first part of the alternative is not true and for $x_2$ the second one is not valid, i.e. $F(a,x_1)>0$ and $F(x_2, b)>0$. Since sections of $F$ are continuous, then there exist two non-void open sets $U_1, U_2 \subset X$ such that $F(a, \cdot)>0$ on $U_1$ and $F(\cdot, b)>0$ on $U_2$. Now, apply alternative \eqref{alt} for all elements of $U_1$ and  $U_2$ to derive the equality $F=0$ on $U_1 \times \{ b\} \cup \{a\} \times U_2$.  
\end{proof}

We will utilize Lemma \ref{FZ} to show that the set $Z$ of zeros of $F$, if it is non-empty, then it is large in some sense. We will use the notion of set ideals. Recall that a family $\mathcal{I} \subset 2^X$ is a set ideal if 
\begin{itemize}
	\item[(a)] $A \in \mathcal{I}$ and $B \subset A$ implies $B \in \mathcal{I}$,
	\item[(b)] $A, B \in \mathcal{I}$ implies $A \cup B \in \mathcal{I}$.
\end{itemize}
We call elements of an ideal small sets, and a set is large if it is not small.
An example of an ideal is the family of all subsets of a topological space having non-empty interior.
Given a set ideal $\mathcal{I}$ of subsets of a set $X$ we define the product ideal $\mathcal{I} \otimes \mathcal{I}$ of subsets of $X \times X$ as the family of all sets $A \subseteq X \times X$ such that
$$\{ x \in X : A[x] \notin \mathcal{I}  \} \in \mathcal{I},$$
where $$A[x]=\{ y \in X : (x,y) \in A \}.$$
For a comprehensive study of the notion of set ideals and several further examples we refer the reader to the monograph of J.C. Oxtoby \cite{O}.

\begin{cor}\label{cI}
Assume that $X$ is a topological space, $\mathcal{I} \subset 2^X$ is a set ideal which does not contain open non-void sets  and  $F\colon X \times X \to [0, + \infty)$ is a continuous solution of \eqref{F} such that the set $Z$ of zeros of $F$ is non-empty and $(a,b) \in Z$ is arbitrary. 
Then $\{a\} \times X \subseteq Z$ or $ X\times \{b\} \subseteq Z $ or $Z$ is a large set with respect to the product ideal $\mathcal{I} \otimes \mathcal{I}$.
\end{cor}
\begin{proof}
Let us pick some $(a,b) \in Z$ arbitrarily and assume that $Z$ does not contain any of the sets $\{a\} \times X$ and $ X\times \{b\} $.
Apply  Lemma \ref{FZ} to obtain $$U_1 \times \{ b\} \cup \{a\} \times U_2 \in Z$$ for some open non-void sets $U_1, U_2 \subset X$. Then, use the same lemma for every point of this set to deduce from the definition of the product ideal that $Z \notin \mathcal{I} \otimes \mathcal{I}$.
\end{proof}

With the same proof one can deduce that in  Lemma \ref{FZ} and Corollary \ref{cI} it is enough to assume that both sections of $F$ are continuous.

Function $F_3$ of Example \ref{e21} is a solution of inequality \eqref{F} for which the set of zeros $Z$  contain both a horizontal and a vertical line and is small with respect to the product ideal $\mathcal{I} \otimes \mathcal{I}$ on $X \times X$ (for every ideal $\mathcal{I}$ satisfying the assumptions of Corollary \ref{cI}).

We will terminate this section with an open problem related to the last statement. 

\begin{pro}
Is it true that under the assumptions of Corollary \ref{cI} if the set $Z$ does not contain a set of the form $\{a\} \times X$ or $ X\times \{b\} $, then it has a non-void interior with respect to the product topology on $X \times X$?
\end{pro}

\section{Additive Sincov's inequality}

\emph{Generalized metric space} or \emph{Lawvere space} (see F.W. Lawvere \cite{L}) is a non-void set $X$ together with a function $H\colon X \times X \to \R$, called a \emph{generalized metric}, which is non-negative, vanishes on $\Delta  = \{ (x,x) : x \in X\} $ and satisfies  the triangle inequality:
\begin{equation}\label{T}
H(x,z) \leq H(x,y) + H(y,z), \quad x, y, z \in X.
\end{equation}
One can find several different names for this notion in the literature. In J. Goubault-Larrecq \cite{JG} it is called \emph{hemi-metric}, whereas in   H.P.A. K\"unzi \cite{K} it is termed \emph{quasi-metric}. Let us note that M.J. Campi\'on, E. Indur\'ain, G. Ochoa and O. Valero \cite{c} studied \emph{weightable quasi-metric} in connection with several functional equations, in particular with additive Sincov's equation.

\medskip

We can apply our results of Section 2 to obtain a characterization of solutions of \eqref{T}. Our settings are fairly general in comparison to the definition of a generalized metric, but instead we assume that we already have a topology on the set $X$.

For an arbitrary function $H\colon X \times X \to \R$ let us define
$$\mathcal{H}(H) = \left\{  \f\colon X \to \R : \forall_{x, y \in X} \, \f(x)-\f(y) \leq H(x,y) \right\}.$$

\begin{cor}\label{cT}
Assume that $X$ is a separable topological space and  $H\colon X \times X \to \R$ is a solution of \eqref{T} which is continuous   and equal to $0$ at every point of  $\Delta$. Then 
\begin{equation}\label{repH}
H(a,b) = \sup \left\{ \f(a)-\f(b) : \f \in \mathcal{H}(H) \right\}, \quad a, b \in X.
\end{equation}
Conversely, for an arbitrary family $\mathcal{H}$ of real functions on $X$  every mapping $H\colon X \times X \to \R$ defined by \eqref{repH} solves \eqref{T}, it is equal to $0$ on $\Delta$ and  $\mathcal{H}\subseteq \mathcal{H}(H)$.
\end{cor}
\begin{proof}
Apply Corollary \ref{csup} for $ G:= \exp \circ H$ and denote $\f = \log \circ f$.
\end{proof}

\begin{cor}\label{cT2}
Under the assumptions of Corollary \ref{cT}, there exist a quotient subspace $X_0$ of $X$ such that:
\begin{itemize}
		\item[(a)] family $ \mathcal{H}(H) $ separates points of $X_0$,
		\item[(b)]  every $\f \in \mathcal{H}(H)$ satisfies a Lipschitz-type condition: 
		$$|\f(a) - \f(b)|\leq \frac12[H(a,b) + H(b,a)], \quad a, b \in X,$$
		\item[(c)] $H|_{X_0\times X_0}(a,b) = 0$ if and only if $a=b$.
\end{itemize}
\end{cor}
\begin{proof}
Introduce an equivalence relation on $X$ as follows. We will write $a \sim b$ whenever $\f(a) = \f(b)$ for every $\f \in \mathcal{H}(H)$. Clearly, this is an equivalence relation. Let $X_0$ be the quotient space with respect to $\sim$. We can embed $X_0$ in $X$ by choosing any representative of each class of abstraction and (a) follows. Point (b) is a direct consequence of \eqref{repH}. 
By Corollary \ref{cT} we get $H(a,b)=0$ whenever $a\sim b$, which proves (c). 
\end{proof}


\medskip

\section{Acknowledgements}
The work performed in this study  was supported by the National Science Centre, Poland (under Grant No. 2015/19/B/ST6/03259).

\end{document}